\documentclass[12pt]{article}

\input diagram.tex
\usepackage{amsmath} % eg. for \begin{equation}
\usepackage{youngtab}
\usepackage[mathscr]{eucal}
\usepackage{amssymb} % eg. for \leqslant
\usepackage{theorem}
\theoremstyle{change} % puts numbers IN FRONT of "Theorem"
\newtheorem{theorem}{Theorem}[section] % defines environment "Theorem".
\newtheorem{lemma}[theorem]{Lemma} % defines environment "Lemma", that
\newtheorem{proposition}[theorem]{Proposition}

\theorembodyfont{\rmfamily} % has the effect that the content of Remark,
\newtheorem{nothing}[theorem]{} % empty Theoremumgebung.

\newenvironment{proof}{\noindent{\bf Proof}\ }{\qed\bigskip}

\renewcommand{\le}{\leqslant}
\renewcommand{\ge}{\geqslant} % needs amssymb-Paket

\renewcommand{\marginpar}[1]{}

\newcommand{\Ctilde}{\widetilde{C}}
\newcommand{\Hom}{\mathrm{Hom}}
\newcommand{\HomhatH}{\Hom^{\wedge}_{\scrH}}
\newcommand{\im}{\mathrm{im}}
\newcommand{\liso}{\mathop{\longrightarrow}\limits^{\sim}}
\newcommand{\NN}{\mathbb{N}}
\newcommand{\qed}{\nobreak\hfill
                   \vbox{\hrule\hbox{\vrule\hbox to 5pt
                   {\vbox to 8pt{\vfil}\hfil}\vrule}\hrule}}
\newcommand{\lliso}{\mathop{\longleftarrow}\limits^{\sim}}
\newcommand{\scrB}{\mathscr{B}}
\newcommand{\scrD}{\mathscr{D}}
\newcommand{\scrDhat}{\scrD^{\wedge}}
\newcommand{\scrDns}{\scrD^{\mathrm{ns}}}
\newcommand{\scrDst}{\scrD^{\mathrm{st}}}
\newcommand{\scrH}{\mathscr{H}}
\newcommand{\scrK}{\mathscr{K}}
\newcommand{\scrP}{\mathscr{P}}
\newcommand{\scrT}{\mathscr{T}}
\newcommand{\scrTbar}{\overline{\scrT}}
\newcommand{\scrThat}{\scrT^{\wedge}}
\newcommand{\scrTrs}{\scrT^{\mathrm{rs}}}
\newcommand{\scrTst}{\scrT^{\mathrm{st}}}
\newcommand{\Xtilde}{\widetilde{X}}
\title{Permutation Resolutions for Specht Modules of Hecke Algebras\footnote{{\bf MR Subject Classification:} 20C08. {\bf Keywords:} Iwahori Hecke algebra, Specht module, permutation module, resolution}}
\author{\small Robert Boltje\\
        \small Department of Mathematics\\
        \small University of California\\
        \small Santa Cruz, CA 95064\\
        \small U.S.A.\\
        \small boltje@ucsc.edu
       \and
        \small Filix Maisch\\
        \small Department of Mathematics\\
        \small Oregon State University\\
        \small Corvallis, OR 97331\\
        \small U.S.A.\\
        \small maischf@science.oregonstate.edu}
\date{October 17, 2011\\ {\small (revised March 25, 2012)}}

\begin{document}

\sloppy

\maketitle

%%%%%%%%%%%%%%%%%%%%%%%%%%%% ABSTRACT %%%%%%%%%%%%%%%%%%%%%%%%%

\begin{abstract}
\noindent 
In \cite{BH}, a chain complex was constructed in a combinatorial way which conjecturally is a resolution of the (dual of the) integral Specht module for the symmetric group in terms of permutation modules. In this paper we extend the definition of the chain complex to the integral Iwahori Hecke algebra and prove the same partial exactness results that were proved in the symmetric group case.
\end{abstract}

%%%%%%%%%%%%%%%%%%%%%% INTRODUCTION %%%%%%%%%%%%%%%%%%%%%%%%%%%%

\section*{Introduction}
In \cite{BH}, Hartmann and the first author constructed, for any composition $\lambda$ of a positive integer $r$, a finite chain complex of modules for the group algebra $RW$ of the symmetric group $W$ on $r$ letters over an arbitrary commutative ring $R$. The last module in this complex is the dual of the Specht module $S^\lambda$ and the other modules are permutation modules with point stabilizers given by Young subgroups of $W$. The construction of the chain complex was completely combinatorial and characteristic-free. It was conjectured that this chain complex is exact whenever $\lambda$ is a partition. In other words the chain complex is conjectured to be a resolution of the (dual) of the Specht module by permutation modules. Partial exactness results were  already established in \cite{BH} and a full proof of the exactness was given recently by Yudin and Santana (see \cite{SY}) by translating the construction via the Schur functor. Other permutation resolutions of Specht modules have been considered by Donkin in \cite{Donkin}, Akin in \cite{Akin}, Zelevinskii in \cite{Zelevinskii}, Akin-Buchsbaum in \cite{AB1} and \cite{AB2}, Santana in \cite{Santana}, Woodcock in \cite{Woodcock1} and \cite{Woodcock2}, Doty in \cite{Doty} and Yudin in \cite{Yudin}. See \cite[Section~6]{BH} for a more detailed comparison of these constructions with the construction in \cite{BH}.

\medskip
The goal of this paper is to (a) lift the construction of the chain complex in \cite{BH} for the group algebra $RW$ to a chain complex for modules of the Iwahori Hecke algebra $\scrH_{r,q}^R$ for any integral domain $R$ such that the specialization $q=1$ reproduces the original chain complex; and (b) lift the partial exactness proofs from \cite{BH} to the new construction.

\medskip
Both goals are achieved, with part (a) being relatively straightforward, but part (b) requiring much more subtle arguments. The paper is arranged as follows. In Section~1 we establish the necessary notation and recall results about the Hecke algebra $\scrH_{r,q}^R$ from \cite{DJ}. In Section~2 we introduce a group of homomorphisms between \lq permutation modules\rq\ of $\scrH_{r,q}^R$ and show that such homomorphisms are closed under composition. This is used in Section~3 to construct the chain complex which generalizes the chain complex in \cite{BH}. In Section~4 we prove that the chain complex in question is exact in degrees $-1$ and $0$, for arbitrary partitions $\lambda$, and that it is exact everywhere, for all partitions of the form $(\lambda_1,\lambda_2,\lambda_3)$ with $\lambda_3\le 1$, cf.~Theorems~\ref{exactness at -1}, \ref{exactness at 0}, \ref{exactness for tame compositions}. The very technical proof of a key lemma, Lemma~\ref{key lemma}, is postponed to Section~5.

\medskip
The authors are grateful to the referee for the careful reading and for pointing out a mistake in the proof of Proposition~\ref{ascending composition proposition} in a first version of this paper.

%%%%%%%%%%%%%%%%% SECTION 1 %%%%%%%%%%%%%%%%%%%%%%%%%%%%%%%%%%%

\section{Notation and Quoted Results}

Throughout this paper we denote by $R$ an integral domain. Unadorned tensor products and homomorphism sets will be understood to be taken over $R$. Moreover, we fix a positive integer $r$ and denote by $W$ the symmetric group on the set $\{1,\ldots,r\}$. Elements of $W$ are composed like functions applied on the right in order to be consistent with \cite{DJ}. Thus, $W$ acts from the right on $\{1,\ldots,r\}$ and we write $(i)w$ or $iw$ for $i\in\{1,\ldots,r\}$ and $w\in W$. We set $S:=\{(i,i+1)\mid i=1\ldots,r-1\}$, the standard choice of simple transpositions. The set of positive (resp.~non-negative) integers will be denoted by $\NN$ (resp.~$\NN_0$).

\begin{nothing}\label{Hecke algebra} {\em The Hecke algebra $\scrH$.}\quad
For a unit $q$ of $R$ we denote by $\scrH=\scrH_{r,q}^R$ the Hecke algebra as introduced in \cite{DJ}. It has an $R$-basis consisting of the elements $T_w$, $w\in W$. The multiplication in $\scrH$ is uniquely determined by the following formulas for $w\in W$ and $s\in S$:
\begin{equation*}
  T_wT_s=
  \begin{cases} T_{ws}, & \text{if $l(ws)=l(w)+1$, i.e., if $iw^{-1}<(i+1)w^{-1}$,}\\
       qT_{ws}+(q-1)T_w, & \text{if $l(ws)=l(w)-1$, i.e., if $iw^{-1}>(i+1)w^{-1}$.}
   \end{cases}
\end{equation*}
The element $T_1$ is the identity element of $\scrH$.

By $\trianglelefteq$ we denote the {\em strong Bruhat order} on $W$: for $u,v\in W$, one has $u\trianglelefteq v$ if there exists a reduced expression $u=s_1\cdots s_n$ of $u$ and integers $1\le i_1<\cdots< i_k\le n$ with $v=s_{i_1}s_{i_2}\cdots s_{i_k}$. While, for given $v,w\in W$, it is in general difficult to express $T_vT_w$ in terms of the basis elements $T_x$, $x\in W$, Lemma~2.1(ii) in \cite{DJ} and a quick induction argument on $l(w)$ imply that
\begin{equation}\label{Bruhat property}
  T_vT_w\in\sum_{w\trianglelefteq u\in W}RT_{vu}
\end{equation}
and that
\begin{equation}\label{lengths add property}
  T_vT_w=T_{vw}\quad\text{if $l(vw)=l(v)+l(w)$.}
\end{equation}
\end{nothing}

\begin{nothing}\label{permutation module}{\em The $\scrH$-modules $M^\lambda$.}\quad
We denote by $\Gamma$ the set of {\em compositions} $\lambda=(\lambda_1,\lambda_2,\ldots)$ of $r$ and by $\Lambda$ the set of partitions of $r$. We define the {\em dominance order} $\trianglelefteq$ on $\Gamma$ by
\begin{equation*}
  \lambda\trianglelefteq\mu\colon\iff \sum_{i=1}^e \lambda_i \le \sum_{i=1}^e \mu_i \quad\text{for all $e\ge 1$.}
\end{equation*}
This defines a partial order on $\Gamma$. Note that it differs from the dominance relation defined in \cite{DJ} which is not a partial order and is defined by
\begin{equation}\label{DJ dominance}
  \lambda\le\mu \colon\iff \mu'\trianglelefteq\lambda'\,,
\end{equation}
where $\lambda'=(\lambda'_1,\lambda'_2,\ldots)$ denotes the {\em dual} of $\lambda$, i.e., $\lambda'_i:=|\{k\in\NN\mid \lambda_k\ge i\}|$, for $i\in\NN$. Note that $\lambda'$ is always a partition.

Let $\lambda$ be a composition of $r$. One associates to $\lambda$ the set partition 
\begin{equation*}
  \scrP_\lambda:=\{ \{1,\ldots,\lambda_1\}, \{\lambda_1+1,\ldots,\lambda_1+\lambda_2\}, \ldots\}
\end{equation*}
of $\{1,\ldots,r\}$ and the {\em standard parabolic subgroup} $W_\lambda$ consisting of all $w\in W$ which fix $\scrP_\lambda$ element-wise. Note that $W_\lambda$ is generated by the elements $s=(i,i+1)\in S$ with the property that $i$ and $i+1$ belong to the same element of $\scrP_\lambda$. More generally, a subgroup of $W$ is called a {\em parabolic subgroup} if it arises as the stabilizer of an arbitrary set partition of $\{1,\ldots,r\}$, i.e., if it is conjugate to a subgroup of the form $W_\lambda$. Intersections of parabolic subgroups are again parabolic subgroups. The right $\scrH$-submodule $M^\lambda$ of the regular module $\scrH$ is defined as
\begin{equation*}
  M^\lambda:=x_\lambda\scrH\quad \text{with } x_\lambda:=\sum_{w\in W_\lambda}T_w\,.
\end{equation*}
For $v\in W_\lambda$ one has
\begin{equation}\label{trivial action}
  x_\lambda T_v = q^{l(v)} x_\lambda\,,
\end{equation}
cf.~\cite[Lemma~3.2]{DJ}. Every coset $W_\lambda w\in W_\lambda\backslash W$ has a unique element of smallest length. We denote the set of these distinguished coset representatives by $\scrD_\lambda$. For $w\in W_\lambda$ and $d\in\scrD_\lambda$ one has
\begin{equation}\label{additivity}
  l(wd)=l(w)+l(d)\quad\text{and}\quad T_{wd}=T_wT_d\,.
\end{equation}
% More generally ...
The elements $x_\lambda T_d$, $d\in\scrD_\lambda$, form an $R$-basis of $M^\lambda$, cf.~\cite[Lemma~3.2]{DJ}.

As usual one identifies compositions $\lambda$ with Young diagrams and we say that $\lambda$ is the {\em shape} of the corresponding Young diagram. A $\lambda$-tableau is a filling of the $r$ boxes of the Young diagram of $\lambda$ with the numbers $1,2,\ldots,r$. We denote the set of $\lambda$-tableaux by $\scrT(\lambda)$ and usually denote elements of $\scrT(\lambda)$ by $t$. The $\lambda$-tableau which contains the entries $1,2,\ldots,r$ in ascending order is denoted by $t^\lambda$. Thus, for $\lambda=(1,0,2)$ one has
\begin{equation*}
  \raise6mm\hbox{$t^{(1,0,2)}=$\ \ }
  \vbox{\noindent \young(1)\\ \yng(0)\\   \young(2)\young(3)}
  \hspace{-12.5cm}
  \raise5mm\hbox{.}
\end{equation*}
This provides a standard numbering of the boxes of the Young diagram of $\lambda$ which we will use later. The group $W$ acts from the right on $\scrT(\lambda)$ by simply applying an element $w\in W$ to the entries of the tableau $t\in\scrT(\lambda)$. This action is free and transitive, and it yields a bijection
\begin{equation*}
  W\liso\scrT(\lambda)\,,\quad w\mapsto t^\lambda w\,.
\end{equation*}
A $\lambda$-tableau $t$ is called {\em row-standard} if its entries are increasing in each row from left to right. The row-standard $\lambda$-tableaux form a subset $\scrTrs(\lambda)$ of $\scrT(\lambda)$. Two $\lambda$-tableaux $t_1$ and $t_2$ are called {\em row-equivalent} if they arise from each other by rearranging elements within each row. We denote the row-equivalence class of $t$ by $\{t\}$ and the set of row-equivalence classes by $\scrTbar(\lambda)$. One has canonical bijections
\begin{equation}\label{d-t-bijection}
  \scrD_\lambda\liso\scrTrs(\lambda)\liso\scrTbar(\lambda)
\end{equation}
given by $d\mapsto t^\lambda d$ and $t\mapsto\{t\}$. Thus, the canonical basis $x_\lambda T_d$, $d\in\scrD_\lambda$, could also be parametrized by $\scrTrs(\lambda)$ or $\scrTbar(\lambda)$. For $d\in\scrD_\lambda$ we denote by 
\begin{equation*}
  \varepsilon_d\in\Hom(M^\lambda,R)
\end{equation*}
the $R$-module homomorphism with the property that $\varepsilon_d(x_\lambda T_e)=\delta_{d,e}$ for all $e\in\scrD_\lambda$. In other words, the elements $\varepsilon_d$, $d\in\scrD_\lambda$, form the dual basis of the $R$-basis $x_\lambda T_d$, $d\in \scrD_\lambda$. If $t\in\scrTrs(\lambda)$ corresponds to $d\in\scrD_\lambda$ under the canonical bijection in (\ref{d-t-bijection}) we will also write $\varepsilon_t$ instead of $\varepsilon_d$.

For $\lambda\in\Lambda$ we have an obvious bijection $\scrT(\lambda)\liso\scrT(\lambda')$, $t\mapsto t'$,
where $t'$ is the reflection of $t$ with respect to the diagonal axis. Note that $(tw)'=t'w$ for all $t\in\scrT(\lambda)$ and $w\in W$.
\end{nothing}

\begin{nothing}\label{Hom(M,M)} {\em The homomorphisms $\varphi_d^{\lambda,\mu}\colon M^\mu\to M^\lambda$.}\quad 
Let $\mu,\lambda\in\Gamma$ be  compositions of $r$. The elements of $\scrD_{\lambda,\mu}:=\scrD_\lambda \cap\scrD_\mu^{-1}$ form a set of representatives of the double cosets $W_\lambda\backslash W/W_\mu$, and each element $d\in\scrD_{\lambda,\mu}$ is the unique element of shortest length in its double coset $W_\lambda d W_\mu$. By \cite[Theorem~3.4]{DJ}, the set $\scrD_{\lambda,\mu}$ parametrizes an $R$-basis $\varphi_d^{\lambda,\mu}$, $d\in\scrD_{\lambda,\mu}$, of $\Hom_\scrH(M^\mu,M^\lambda)$ given by
\begin{equation}\label{explicit permutation homomorphism}
  \varphi_d^{\lambda,\mu}(x_\mu)= \sum_{w\in W_\lambda d W_\mu} T_w = 
  x_\lambda \sum_{e\in \scrD_\nu\cap W_\mu} T_{de} = 
  x_\lambda T_d \sum_{e\in\scrD_{\nu}\cap W_\mu} T_e\,,
\end{equation}
where $\nu\in\Gamma$ is determined by $W_\nu=d^{-1}W_\lambda d \cap W_\mu$. The last equality follows from (\ref{additivity}). We want to mention that the cited theorem requires that $R$ is a principal ideal domain. But a careful examination of the canonical isomorphisms in \cite[Theorems~2.5, 2.6, 2.7]{DJ} involved in the proof of the theorem show that this hypothesis is unnecessary. (More precisely, the $R$-modules $\Hom_\scrH(M_{\scrH_\nu},N\otimes T_d)$ in \cite[Theorem~2.8]{DJ} are free of rank one in the case that $M$ and $N$ are trivial modules.)

We will use a combinatorial description of the set $\scrD_{\lambda,\mu}$ in terms of the set $\scrT(\lambda,\mu)$ of generalized tableaux $T$ of shape $\lambda$ and content $\mu$. Such a generalized tableau $T$ is a filling of the $r$ boxes of the Young diagram of shape $\lambda$ with $\mu_1$ entries equal to $1$, $\mu_2$ entries equal to $2$, etc. We denote by $T^\lambda_\mu\in\scrT(\lambda,\mu)$ the generalized tableau which has its boxes filled in the natural order. If we number the boxes of the Young diagram of $\lambda$ according to the  entries of $t^\lambda\in\scrT(\lambda)$, we may view $\scrT(\lambda,\mu)$ as the set of functions $T\colon \{1,\ldots, r\}\to\NN$ with the property $|T^{-1}(i)|=\mu_i$ for all $i\in\NN$. The group $W$ acts transitively on $\scrT(\lambda,\mu)$ from the left by $(wT)(i):=T(iw)$, for $T\in\scrT(\lambda,\mu)$, $w\in W$ and $i\in\{1,\ldots,r\}$. The stabilizer of $T_\mu^\lambda$ is equal to $W_\mu$. This defines a bijection $W/W_\mu\liso\scrT(\lambda,\mu)$, $wW_\mu\mapsto wT^\lambda_\mu$. We call two generalized tableaux $T_1,T_2\in\scrT(\lambda,\mu)$ {\em row-equivalent} if they arise from each other by rearranging the entries within the rows, i.e., if $T_2=wT_1$ for some $w\in W_\lambda$. If we denote the row equivalence class of $T\in\scrT(\lambda,\mu)$ by $\{T\}$  and the set of such classes by $\scrTbar(\lambda,\mu)$ then we have obtain a bijection $W_\lambda\backslash W/W_\mu \liso \scrTbar(\lambda,\mu)$, $W_\lambda w W_\mu\mapsto \{wT^\lambda_\mu\}$. A generalized tableau $T\in\scrT(\lambda,\mu)$ is called {\em row-semistandard} if in each of its rows the entries are in their natural order from left to right. We denote the set of these tableaux by $\scrTrs(\lambda,\mu)$. Each row-equivalence class contains a unique row-semistandard element. Thus, $\scrTrs(\lambda,\mu)\to\scrTbar(\lambda,\mu)$, $T\mapsto\{T\}$, is a bijection. Altogether, we now have canonical bijections
\begin{equation}\label{d-T-bijection}
  \scrD_{\lambda,\mu}\liso W_\lambda\backslash W/W_\mu \liso \scrTbar(\lambda,\mu) \lliso \scrTrs(\lambda,\mu)
\end{equation}
and we may use each of these sets to parametrize the basis $\varphi_d^{\lambda,\mu}$, $d\in\scrD_{\lambda,\mu}$, of $\Hom_\scrH(M^\mu, M^\lambda)$. So, if $T\in\scrTrs(\lambda,\mu)$ and $d\in\scrD_{\lambda,\mu}$ correspond under the above bijection, we also write $\varphi_T^{\lambda,\mu}$ instead of $\varphi_d^{\lambda,\mu}$.

We leave it to the reader to check that $d\in\scrD_{\lambda,\mu}$ corresponds to $T\in\scrTrs(\lambda,\mu)$ if and only if $\varphi_T^{\lambda,\mu}=\theta_T\colon M^\mu\to M^\lambda$ in the case $q=1$, with $\theta_T$ the homomorphism defined in \cite[Subsection~1.5]{BH}.
\end{nothing}

\begin{nothing}\label{Specht modules} {\em The Specht modules $S^\lambda$.}\quad
Let $\lambda\in\Lambda$ be a partition of $r$. One defines elements
\begin{equation*}
  y_\lambda:=\sum_{w\in W_\lambda} (-q)^{-l(w)} T_w\quad\text{and}\quad
  z_\lambda:=x_\lambda T_{w_\lambda} y_{\lambda'} \in M^\lambda
\end{equation*}
of $\scrH$, where $w_\lambda\in W$ is defined by the property that $(t^\lambda w_\lambda)'=t^{\lambda'}$. The {\em Specht module} $S^\lambda$ is defined by
\begin{equation*}
  S^\lambda:=z_\lambda \scrH\,.
\end{equation*}
Since $z_\lambda\in M^\lambda$, $S^\lambda$ is an $\scrH$-submodule of $M^\lambda$.

A $\lambda$-tableau $t\in\scrT(\lambda)$ is called a {\em standard} $\lambda$-tableau if $t\in\scrTrs(\lambda)$ and $t'\in\scrTrs(\lambda')$. We denote the set of standard $\lambda$-tableaux by $\scrTst(\lambda)$. More generally, for $\mu\in\Gamma$, we say that a generalized tableau $T\in\scrT(\lambda,\mu)$ is {\em standard} if $T\in\scrTrs(\lambda,\mu)$ and if the entries of $T$ are strictly increasing along each column from top to bottom. We denote the set of standard $\lambda$-tableaux of content $\mu$ by $\scrTst(\lambda,\mu)$.

We extend the definition of $S^\lambda$, $\scrTst(\lambda)$, and $\scrTst(\lambda,\mu)$ to arbitrary compositions $\lambda\in\Gamma$ by setting $S^\lambda:=0$, $\scrTst(\lambda)=\emptyset$, and $\scrTst(\lambda,\mu)=\emptyset$ if $\lambda$ is not a partition.

Recall that the {\em weak Bruhat order} $\le$ on $W$ is defined by $u\ge v$ if and only if there exists a reduced expression $v=s_1\cdots s_n$ for $v$ with $s_1,\ldots,s_n\in S$ and $k\in\{0,\ldots,n\}$ such that $u=s_1\cdots s_k$. Given $\lambda\in\Lambda$, \cite[Lemma~1.5]{DJ} states that
\begin{equation}\label{weak Bruhat 1}
  \scrDst_\lambda:=\{d\in W\mid d\ge w_\lambda\}\liso \scrTst(\lambda)\,,\quad d\mapsto t^\lambda d\,.
\end{equation}
is a bijection.
Thus, $\scrDst_\lambda\subseteq \scrD_\lambda$. Note that for $d\in W$ one has $t^{\lambda'}d = (t^\lambda w_\lambda)'d = (t^\lambda w_\lambda d)'$ and therefore,
\begin{equation}\label{weak Bruhat 2}
  d\ge w_{\lambda'} \iff t^{\lambda'}d\in\scrTst(\lambda') \iff t^\lambda w_\lambda d\in \scrTst(\lambda) \iff
  w_\lambda d\in\scrDst_\lambda\,.
\end{equation}
By \cite[Theorem~5.6]{DJ}, for given $\lambda\in\Lambda$, the elements $z_\lambda T_d$ with $w_{\lambda'}\le d\in W$ form an $R$-basis of $S^\lambda$. By the above, this set has cardinality $|\scrTst(\lambda')|=|\scrTst(\lambda)|$, since $t\in\scrT(\lambda)$ is standard if and only if $t'\in\scrT(\lambda')$ is standard. Moreover, \cite[Lemma~5.1]{DJ} states that  if $d\ge w_{\lambda'}$ and if one expands the basis element $z_\lambda T_d$ of $S^\lambda$ in terms of the standard basis of $M^\lambda$, i.e., $z_\lambda T_d = \sum_{e\in\scrD_\lambda} \alpha_{d,e} x_\lambda T_e$, with $\alpha_{d,e}\in R$, then
\begin{equation}\label{Specht property 1}
  \alpha_{d,w_\lambda d} = q^{l(d)}
\end{equation}
and
\begin{equation}\label{Specht property 2}
  \alpha_{d,e}\neq 0 \Rightarrow \bigl(e=w_\lambda d \text{ or } l(e)>l(w_\lambda d)\bigr)\,.
\end{equation}
Thus, if we set $\scrDns_\lambda:=\scrD_\lambda\smallsetminus\scrDst_\lambda$ then the $R$-span of the elements $x_\lambda T_e$, $e\in\scrDns_\lambda$, is a direct complement of $S^\lambda$ in $M^\lambda$, and we state for later reference:
\begin{equation}\label{Specht property 3}
  \parbox{11.2cm}{\it
    For every $\lambda\in\Gamma$, the Specht module $S^\lambda$ is $R$-free of rank $|\scrTst(\lambda)|$ and $S^\lambda$ has an $R$-complement in $M^\lambda$.}
\end{equation}

\end{nothing}

%%%%%%%%%%%%%%%%%% SECTION 2 %%%%%%%%%%%%%%%%%%%%%%%%%%%%%%%%%%%

\section{Ascending Generalized Tableaux and Homomorphisms}\label{sec Homhat}

In this section we introduce, for any compositions $\lambda,\mu\in\Gamma$, an $R$-submodule $\HomhatH(M^\mu,M^\lambda)$ of $\Hom_{\scrH}(M^\mu,M^\lambda)$. In Proposition~\ref{ascending composition proposition} we show that for $\varphi\in\HomhatH(M^\mu,M^\lambda)$ and $\psi\in\HomhatH(M^\nu,M^\mu)$ one has $\varphi\circ \psi\in\HomhatH(M^\nu,M^\lambda)$.

\begin{nothing}\label{ascending Tableaux}
Let $\lambda,\mu\in\Gamma$ be compositions of $r$. We say that $T\in\scrT(\lambda,\mu)$ is {\em ascending} if, for every $i\in\NN$, the $i$-th row of $T$ contains only entries which are greater than or equal to $i$. As in \cite{BH}, we denote the set of ascending and row semistandard elements of $\scrT(\lambda,\mu)$ by $\scrThat(\lambda,\mu)$. One has
\begin{equation}\label{ascending existence}
  \scrThat(\lambda,\mu)\neq\emptyset \iff \mu\trianglelefteq \lambda \iff T^\lambda_\mu\in\scrThat(\lambda,\mu)
\end{equation}
and $\scrThat(\lambda,\lambda)=\{T^\lambda_\lambda\}$.

We define $\scrDhat_{\lambda,\mu}$ as the image of $\scrThat(\lambda,\mu)$ under the canonical bijection in (\ref{d-T-bijection}). Moreover, we define
\begin{equation}\label{eqn Homhat definition}
  \HomhatH(M^\mu,M^\lambda):=\bigoplus_{d\in\scrDhat_{\lambda,\mu}} R\varphi_d^{\lambda,\mu} \subseteq
  \Hom_{\scrH}(M^\mu,M^\lambda)\,.
\end{equation}
By (\ref{ascending existence}) we have
\begin{equation*}
  \HomhatH(M^\mu,M^\lambda)\neq\{0\} \Rightarrow \mu\trianglelefteq \lambda\,.
\end{equation*}

For $i\in\NN$, we define the intervals $P_i:=\{\lambda_1+\cdots+\lambda_{i-1}+1,\ldots,\lambda_1+\cdots+\lambda_i\}$ and $Q_i:=\{\mu_1+\cdots+\mu_{i-1}+1,\ldots,\mu_1+\cdots+\mu_i\}$. One has for any $w\in W$:
\begin{align*}
  & \text{$wT^\lambda_\mu$ is ascending}\\
  \iff & \text{$P_i w \subseteq Q_i\cup Q_{i+1}\cup \cdots$ for all $i\ge 1$}\\
  \iff & \text{$(P_i\cup P_{i+1}\cup \cdots) w \subseteq Q_i\cup Q_{i+1}\cup \cdots$ for all $i\ge 1$.}
\end{align*}
We define $W_{\lambda,\mu}$ to be the set of elements $w\in W$ satisfying the above equivalent properties. Clearly, $W_{\lambda,\mu}$ is a union of double cosets in $W_\lambda\backslash W/W_\mu$. The following properties are also immediate from the definition:
\begin{align}\label{ascending W-properties}
\begin{split}
  W_{\lambda,\mu} & \neq\emptyset\quad \text{if and only if $\mu\trianglelefteq\lambda$,}\\
  W_{\lambda,\lambda} & = W_\lambda\,,\\
  W_{\lambda,\mu} W_{\mu,\nu} & \subseteq W_{\lambda,\nu} 
              \quad\text{if $\nu\trianglelefteq\mu\trianglelefteq\lambda$,}\\
  W_\lambda 1 W_\mu & \subseteq W_{\lambda,\mu}\quad \text{if $\mu\trianglelefteq \lambda$,}\\
  W_{\lambda,\mu}\cup W_{\mu,\nu} & \subseteq W_{\lambda,\nu}
              \quad\text{if $\nu\trianglelefteq\mu\trianglelefteq\lambda$.}
\end{split}
\end{align}
\end{nothing}

The goal of this section is the proof of the following proposition.

\begin{proposition}\label{ascending composition proposition}
Let $\lambda,\mu,\nu\in\Gamma$ be compositions of $r$ and let $\alpha\in\HomhatH(M^\nu,M^\mu)$ and $\beta\in\HomhatH(M^\mu,M^\lambda)$. Then $\beta\circ\alpha\in\HomhatH(M^\nu,M^\lambda)$.
\end{proposition}

Before we can prove the proposition, we need three lemmas.

\begin{lemma}\label{ascending lemma 1}
Let $\lambda,\mu\in\Gamma$ be compositions of $r$ and let $w\in W_{\lambda,\mu}$ (in particular, $\mu\trianglelefteq\lambda$). If $w=sv$ with $s\in S$, $v\in W$ and $l(w)=l(v)+1$ then there exists $\rho\in\Gamma$ such that $\mu\trianglelefteq\rho\trianglelefteq\lambda$, $s\in W_{\lambda,\rho}$ and $v\in W_{\rho,\mu}$.
\end{lemma}

\begin{proof}
Write $s=(i,i+1)$ with $i\in\{1,\ldots,r-1\}$ and note that $iw>(i+1)w$, since $l(w)>l(sw)$. Let $P_1,P_2,\ldots$ and $Q_1,Q_2,\ldots$ be the subsets of $\{1,\ldots,r\}$ associated to $\lambda$ and $\mu$, respectively, as above. If there exists $j\in \NN$ with $\{i,i+1\}\in P_j$ then $s\in W_\lambda$ and $\rho:=\lambda$ satisfies the required conditions, since $W_\lambda=W_{\lambda,\lambda}$ and $v=sw\in W_\lambda W_{\lambda,\mu}=W_{\lambda,\mu}$. Hence, we may assume that there exist positive integers $j<k$ such that $i\in P_j$ and $i+1\in P_k$ (we need to allow the possibility that $P_{j+1}=\cdots=P_{k-1}=\emptyset$). Set $P'_l:=P_l$ for $l\notin\{j,k\}$, and set $P'_j:=P_j\smallsetminus\{i\}$ and $P'_k:=P_k\cup\{i\}$. Moreover, let $\rho\in\Gamma$ be the composition defined by $P'_1,P'_2,\ldots$. Then clearly, $\rho\trianglelefteq\lambda$ and $s\in W_{\lambda,\rho}$. We still need to show that $v\in W_{\rho,\mu}$. This also implies $\mu\trianglelefteq\rho$ by (\ref{ascending W-properties}). For $l\notin\{j+1,\ldots,k\}$ we have $P'_l\cup P'_{l+1}\cup\cdots = P_l\cup P_{l+1}\cup\cdots$ and $(P'_l\cup P'_{l+1}\cup\cdots)v=(P_l\cup P_{l+1}\cup\cdots)v=(P_l\cup P_{l+1}\cup\cdots)sv = (P_l\cup P_{l+1}\cup\cdots)w\subseteq Q_l\cup Q_{l+1}\cup\cdots$. For $l\in\{j+1,\ldots,k\}$ we have $i+1\in P'_l\cup P'_{l+1}\cup\cdots = (P_l\cup P_{l+1}\cup\cdots)\cup\{i\}$  and $(P'_l\cup P'_{l+1}\cup\cdots)v = (\{i\}\cup P_l\cup P_{l+1}\cup\cdots)v = (\{i\}\cup P_l\cup P_{l+1}\cup\cdots)sv = (\{i\}\cup P_l\cup P_{l+1}\cup\cdots)w \subseteq \{iw\}\cup Q_l\cup Q_{l+1}\cup\cdots$, since $w\in W_{\lambda,\mu}$. It suffices to show that $iw\in Q_l\cup Q_{l+1}\cup\cdots$. But since $w\in W_{\lambda,\mu}$, we have $(i+1)w\in (P_k\cup P_{k+1}\cup\cdots)w\subseteq Q_k\cup Q_{k+1}\cup\cdots \subseteq Q_l\cup Q_{l+1}\cup\cdots$, and $iw>(i+1)w$ now implies also $iw\in Q_l\cup Q_{l+1}\cup\cdots$. This completes the proof of the lemma.
\end{proof}

\begin{lemma}\label{ascending lemma 2}
Let $\lambda,\mu\in\Gamma$ and let $w\in W_{\lambda,\mu}$. If $u\in W$ satisfies $u\trianglerighteq w$ then also $u\in W_{\lambda,\mu}$.
\end{lemma}

\begin{proof}
We proceed by induction on $l(w)$. Note that $\mu\trianglelefteq\lambda$, since $W_{\lambda,\mu}\neq\emptyset$. If $w=1$ we have $u=1$ and $u=w\in W_{\lambda,\mu}$. Now assume that $l(w)\ge 1$ and write $w=sv$ with $l(w)=l(v)+1$ and $s\in S$. By Lemma~\ref{ascending lemma 1} there exists $\rho\in\Gamma$ such that $\mu\trianglelefteq\rho\trianglelefteq\lambda$, $s\in W_{\lambda,\rho}$ and $v\in W_{\rho,\mu}$. Since 
$u\trianglerighteq w=sv$ and $l(w)>l(v)$, we have $u\trianglerighteq v$ or $su\trianglerighteq v$. If $u\trianglerighteq v$ then $l(u)\le l(v)<l(w)$ and, by induction, $v\in W_{\rho,\mu}$ implies $u\in W_{\rho,\mu}$. But $W_{\rho,\mu}\subseteq W_{\lambda,\mu}$ by (\ref{ascending W-properties}). If $su\trianglerighteq v$ then, again by induction, $v\in W_{\rho,\mu}$ implies $su\in W_{\rho,\mu}$, and further, $u=s(su)\in W_{\lambda,\rho}W_{\rho,\mu}\subseteq W_{\lambda,\mu}$. Now the proof is complete.
\end{proof}

\begin{lemma}\label{ascending lemma 3}
Let $\lambda,\mu,\nu\in\Gamma$, $v\in W_{\lambda,\mu}$, and $w\in W_{\mu,\nu}$. Then $T_vT_w\in\sum_{u\in W_{\lambda,\nu}} RT_u$.
\end{lemma}

\begin{proof}
This follows immediately from Equation~(\ref{Bruhat property}), Lemma~\ref{ascending lemma 2} and $W_{\lambda,\mu}W_{\mu,\nu}\subseteq W_{\lambda,\nu}$.
\end{proof}

\noindent
{\bf Proof of Proposition~\ref{ascending composition proposition}}\quad
By the definition in (\ref{eqn Homhat definition}) we may assume that $\alpha=\varphi_d^{\mu,\nu}$ and $\beta=\varphi_e^{\lambda,\mu}$ for some $d\in\scrDhat_{\mu,\nu}$ and $e\in\scrDhat_{\lambda,\mu}$. By Equation~(\ref{explicit permutation homomorphism}) and since $d\in W_{\mu,\nu}$ and $e\in W_{\lambda,\mu}$, we have $\varphi_d^{\mu,\nu}(x_\nu) \subseteq x_\mu\cdot\sum_{w\in W_{\mu,\nu}} RT_w$ and $\varphi_e^{\lambda,\mu}(x_\mu)\subseteq \sum_{v\in W_{\lambda,\mu}} RT_v$. Together with Lemma~\ref{ascending lemma 3} this implies
\begin{align*}
  (\varphi_e^{\lambda,\mu}\circ\varphi_d^{\mu,\nu})(x_\nu) 
  & \in \varphi_e^{\lambda,\mu}(x_\mu)\cdot \sum_{w\in W_{\mu,\nu}} RT_w
  \subseteq \sum_{v\in W_{\lambda,\mu} \atop w\in W_{\mu,\nu}}RT_vT_w\\
  &\subseteq\sum_{u\in W_{\lambda,\nu}} RT_u = \sum_{f\in\scrDhat_{\lambda,\nu}} 
  \sum_{u\in W_\lambda f W_\nu} RT_u\,.
\end{align*}
Since the elements $\varphi_f^{\lambda,\nu}$, $f\in\scrD_{\lambda,\nu}$, form an $R$-basis of $\Hom_{\scrH}(M^\nu,M^\lambda)$ we can write $\varphi_e^{\lambda,\mu}\circ\varphi_d^{\mu,\nu} = \sum_{f\in\scrD_{\lambda,\nu}} a_f\varphi_f^{\lambda,\nu}$ with $a_f\in R$, for $f\in \scrD_{\lambda,\nu}$. Thus,
\begin{equation*}
  (\varphi_e^{\lambda,\mu}\circ\varphi_d^{\mu,\nu})(x_\nu) 
  = \sum_{f\in\scrD_{\lambda,\nu}} a_f \varphi_f^{\lambda,\nu}(x_\nu)
  = \sum_{f\in\scrD_{\lambda,\nu}} a_f \sum_{u\in W_\lambda f W_\nu}T_u\,.
\end{equation*}
Comparing this with the above yields $a_f=0$ for all $f\in\scrD_{\lambda,\nu}\smallsetminus\scrDhat_{\lambda,\nu}$, and the proof is complete.\qed

%%%%%%%%%%%%%%%%%% SECTION 3 %%%%%%%%%%%%%%%%%%%%%%%%%%%%%%%%%%%

\section{The Chain Complex $\Ctilde^\lambda_*$}

Throughout this section we fix a composition $\lambda\in\Gamma$. We will use the $R$-modules $\HomhatH(M^\mu,M^\lambda)$ from Section~\ref{sec Homhat} and the result from Proposition~\ref{ascending composition proposition} to construct a chain complex $\Ctilde_*^\lambda$ of left $\scrH$-modules. In the case $q=1$, this chain complex coincides with the chain complex constructed in \cite{BH} for the symmetric group algebra.

\begin{nothing}\label{chain complex definition} {\em The definition of $\Ctilde^\lambda_*$.}\quad
For every strictly ascending chain
\begin{equation}\label{chain}
  \gamma = (\lambda^{(0)}\triangleleft\cdots\triangleleft\lambda^{(n)})
\end{equation}
of {\em length} $n$ in $\Gamma$ we set
\begin{equation*}
\begin{split}
  M_\gamma:= \ & \HomhatH(M^{\lambda^{(0)}},M^{\lambda^{(1)}}) \otimes
                           \HomhatH(M^{\lambda^{(1)}},M^{\lambda^{(2)}}) \otimes \cdots \\
             & \cdots \otimes \HomhatH(M^{\lambda^{(n-1)}}, M^{\lambda^{(n)}}) \otimes 
                           \Hom(M^{\lambda^{(n)}},R)\,.
\end{split}
\end{equation*}
We view $M_\gamma$ as left $\scrH$-module via
\begin{equation*}
  h(\varphi_1\otimes\cdots\otimes\varphi_n\otimes\varepsilon) :=
  \varphi_1\otimes\cdots\otimes\varphi_n\otimes h\varepsilon\,,
\end{equation*}
for $h\in\scrH$, $\varphi_i\in\HomhatH(M^{\lambda^{(i-1)}},M^{\lambda^{(i)}})$, $i=1,\ldots,n$, and $\varepsilon\in\Hom(M^{\lambda^{(n)}},R)$, where $(h\varepsilon)(m):=\varepsilon(mh)$ for $m\in M^{\lambda^{(n)}}$.

For every integer $n\ge 0$, we write $\Delta_n^\lambda$ for the set of chains $\gamma$ as in (\ref{chain}) with $\lambda^{(0)}=\lambda$. Further, we define the left $\scrH$-module
\begin{equation*}
  C_n^\lambda:= \bigoplus_{\gamma\in \Delta^\lambda_n} M_\gamma\,.
\end{equation*}
For $n\ge 1$, $\gamma\in\Delta_n^\lambda$ and $i\in\{1,\ldots,n\}$, we denote by $\gamma_i\in \Delta_{n-1}^\lambda$ the chain obtained from $\gamma$ by omitting $\lambda^{(i)}$. We define an $\scrH$-module homomorphism
\begin{equation*}
  d_{n,i}^\lambda\colon M_\gamma\to M_{\gamma_i}
\end{equation*}
by
\begin{equation*}
  \varphi_1\otimes\cdots\otimes\varphi_n\otimes\varepsilon\mapsto
  \varphi_1\otimes\cdots\otimes\varphi_{i+1}\circ\varphi_i\otimes\cdots\otimes \varphi_n\otimes\varepsilon\,,
\end{equation*}
where $\varphi_{i+1}$ is interpreted as $\varepsilon$ if $i=n$. The direct sum of these homomorphisms yields an $\scrH$-module homomorphism $d_{n,i}^\lambda\colon C_n^\lambda\to C_{n-1}^\lambda$ and we set
\begin{equation*}
  d_n^\lambda:=\sum_{i=1}^n (-1)^{i-1} d_{n,i}^\lambda\colon C_n^\lambda\to C_{n-1}^\lambda\,.
\end{equation*}
Since the maps $d_{n,i}^\lambda\colon C_n^\lambda\to C_{n-1}^\lambda$ satisfy the usual simplical relations, we obtain $d_n^\lambda\circ d_{n+1}^\lambda =0$ for $n\ge 1$. Thus, we have constructed a chain complex
\begin{equation*}
  C_*^\lambda\colon\quad 0\ar C_{a(\lambda)}^\lambda \Ar{d_{a(\lambda)}^\lambda} C_{a(\lambda)-1}^\lambda
  \Ar{d_{a(\lambda)-1}^\lambda}\cdots\Ar{d_1^\lambda} C_0^\lambda \ar 0
\end{equation*}
of finitely generated left $\scrH$-modules and $\scrH$-module homomorphisms. Here, $a(\lambda)$ is defined as the length of the longest possible strictly ascending chain $\gamma$ as in (\ref{chain}) with $\lambda^{(0)}=\lambda$, in other words, $a(\lambda)$ is the largest integer $n$ with $\Delta_n^\lambda\neq\emptyset$.

Note that the definition of $d_n^\lambda$ implies immediately that
\begin{equation}\label{d-recursion}
\begin{split}
  d_n^{\lambda^{(0)}}(\varphi_1\otimes\cdots\otimes\varphi_n\otimes\varepsilon) = 
  &\ (\varphi_2\circ\varphi_1)\otimes\cdots\otimes\varphi_n\otimes\varepsilon\\
  &\ - \varphi_1\otimes d_{n-1}^{\lambda^{(1)}}(\varphi_2\otimes\cdots\otimes
  \varphi_n\otimes\varepsilon)\,,
\end{split}
\end{equation}
for $n\ge 2$, $\gamma$ as in (\ref{chain}), $\varphi_i\in\HomhatH(M^{\lambda^{(i-1)}},M^{\lambda^{(i)}})$, $i=1,\ldots,n$, and $\varepsilon\in\Hom(M^{\lambda^{(n)}},R)$.

Finally, we extend the chain complex $C_*^\lambda$ by the $\scrH$-module homomorphism
\begin{equation*}
  d_0^\lambda\colon C_0^\lambda = \Hom(M^\lambda,R) \to \Hom(S^\lambda,R)=:C_{-1}^\lambda\,,\quad
  \varepsilon\mapsto\varepsilon|_{S^\lambda}\,,
\end{equation*}
and obtain a chain complex
\begin{equation*}
  \Ctilde_*^\lambda\colon\quad   
  0\ar C_{a(\lambda)}^\lambda \Ar{d_{a(\lambda)}^\lambda} C_{a(\lambda)-1}^\lambda
  \Ar{d_{a(\lambda)-1}^\lambda}\cdots\Ar{d_1^\lambda} C_0^\lambda \Ar{d_0^\lambda} C_{-1}^\lambda\ar 0
\end{equation*}
in the category of finitely generated left $\scrH$-modules. In fact, in the next proposition we show that $d_0^\lambda\circ d_1^\lambda=0$. Note that $C_{-1}^\lambda=0$ if $\lambda$ is not a partition. Also note that every $\scrH$-module in the chain complex $\Ctilde_*^\lambda$ is finitely generated and free as $R$-module so that after applying the functor $\Hom(-,R)$ we obtain a chain complex starting with $0\to S^\lambda\to M^\lambda\to\cdots$ and involving direct sums of modules of the form $M^\mu$ with $\mu\triangleright \lambda$ from there on.
\end{nothing}

\begin{proposition}\label{d0d1=0}
With the notation from \ref{chain complex definition} one has $d_0^\lambda\circ d_1^\lambda=0$.
\end{proposition}

\begin{proof}
We may assume that $\lambda$ is a partition, since otherwise $C_{-1}^\lambda=0$. Let $\lambda\triangleleft\mu$ be a chain of length $1$ in $\Gamma$ and let $d\in\scrD_{\mu,\lambda}$. It suffices to show that $\varphi_d^{\mu,\lambda}(S^\lambda)=0$. Since $S^\lambda=z_\lambda\scrH$, it suffices to show that $\varphi_d^{\mu,\lambda}(z_\lambda)=0$. But,
\begin{equation*}
  \varphi_d^{\mu,\lambda}(z_\lambda) = \varphi_d^{\mu,\lambda}(x_\lambda T_{w_\lambda} y_{\lambda'})
  = \varphi_d^{\mu,\lambda}(x_\lambda) T_{w_\lambda} y_{\lambda'} \in x_\mu\scrH y_{\lambda'}
\end{equation*}
and it suffices to show that $x_\mu T_w y_{\lambda'}=0$ for all $w\in W$. So assume that $x_\mu T_w y_{\lambda'}\neq 0$ for some $w\in W$. Then \cite[Lemma~4.1]{DJ} implies $\lambda''\ge\mu$, with respect to relation $\le$ defined in (\ref{DJ dominance}). But $\lambda''\ge \mu$ is equivalent to $\lambda'''\trianglelefteq\mu'$. Now $\lambda'''=\lambda'$ and $\lambda'\trianglelefteq \mu'$ implies $\mu''\trianglelefteq\lambda''$. Since $\lambda$ is a partition, we have $\lambda''=\lambda$ and obtain $\mu\trianglelefteq\mu''\trianglelefteq\lambda$, a contradiction to $\lambda\triangleleft\mu$. This shows that $x_\mu T_w y_{\lambda'}=0$ for all $w\in W$ and the proof is complete.
\end{proof}

\begin{nothing}\label{basis of Cn}
For $\lambda\in\Gamma$ and an integer $n\ge0$ we denote by $\scrB_n^\lambda$ the set of symbols
\begin{equation}\label{symbol}
  (\lambda^{(0)}\mathop{\triangleleft}\limits_{T_1}
  \lambda^{(1)}\mathop{\triangleleft}\limits_{T_2}\cdots
  \mathop{\triangleleft}\limits_{T_n}\lambda^{(n)}, t)
\end{equation}
with $ \gamma = (\lambda^{(0)}\triangleleft\cdots\triangleleft\lambda^{(n)})\in\Delta_n^\lambda$ (so $\lambda^{(0)}=\lambda$) and $T_i\in\scrThat(\lambda^{(i)},\lambda^{(i-1)})$ for $i=1,\ldots,n$, and with $t\in\scrTrs(\lambda^{(n)})$. By the definition of $C_n^\lambda$, the elements of $\scrB_n^\lambda$ parametrize and $R$-basis of $C_n^\lambda$, by associating the symbol in (\ref{symbol}) with the element 
\begin{equation*}
  \varphi^{\lambda^{(1)},\lambda^{(0)}}_{T_1}\otimes\cdots\otimes
  \varphi^{\lambda^{(n)},\lambda^{(n-1)}}_{T_n}\otimes\varepsilon_t\in M_\gamma\subseteq C_n^\lambda\,.
\end{equation*}
Note that $\scrB_0^\lambda=\scrTrs(\lambda)$. For completeness we set $\scrB_{-1}^\lambda:=\scrTst(\lambda)$ and associate to a standard tableau $t\in\scrTst(\lambda)$ the element $\psi_t\in\Hom(S^\lambda,R)$ with the property $\psi_t(z_\lambda T_d)=1$ if $d\in\scrDst_\lambda$ corresponds to $t$ under the bijection (\ref{weak Bruhat 1}) and $\psi_t(z_\lambda T_d)=0$ otherwise.
\end{nothing}
  
%%%%%%%%%%%%%%%%%% SECTION 4 %%%%%%%%%%%%%%%%%%%%%%%%%%%%%%%%%%%

\section{Exactness Results}

The goal of this section is to state three exactness results, Theorems~\ref{exactness at -1}, \ref{exactness at 0} and \ref{exactness for tame compositions}, on the chain complex $\Ctilde^\lambda_*$. The lengthy and technical proof of a key lemma, Lemma~\ref{key lemma}, which is needed in the proof of the two latter theorems, will be postponed to the next section. The strategy of the proofs is the same as in \cite{BH}. The proofs in this section are adaptations of the proofs in \cite{BH}. We will present them for the reader's convenience in the Hecke algebra setting. However, the proof of Lemma~\ref{key lemma} is more difficult if $q\neq 1$.

\begin{nothing}\label{exactness proof strategy}
Assume that $\lambda\in\Gamma$, that $n\ge-1$ is an integer and that $K^\lambda_n$ is an $R$-submodule of $C_n^\lambda$ satisfying
\begin{equation}\label{A}\tag{$A_n^\lambda$}
  \im(d_{n+1}^\lambda)+K_n^\lambda=C_n^\lambda
\end{equation}
and
\begin{equation}\label{B}\tag{$B_n^\lambda$}
  \ker(d_n^\lambda)\cap K_n^\lambda = \{0\}\,.
\end{equation}
Then it is easy to see that the chain complex $\Ctilde_*^\lambda$ is exact in degree $n$. In fact, the conditions 
($A_n^\lambda$) and ($B_n^\lambda$) together are equivalent to
\begin{equation*}
  \im(d_{n+1}^\lambda) = \ker(d_n^\lambda)\quad\text{and}\quad 
  \ker(d_n^\lambda)\oplus K_n^\lambda = C_n^\lambda\,.
\end{equation*}
We will produce, for certain choices of $\lambda$ and $n$, submodules $K_n^\lambda$ which satisfy ($A_n^\lambda$) and ($B_n^\lambda$), by taking the $R$-span of the basis elements of $C_n^\lambda$ parametrized by a subset $\scrK_n^\lambda$ of the canonical $R$-basis $\scrB_n^\lambda$, cf.~Subsection~\ref{basis of Cn}. By abuse of notation we then will say that $\scrK_n^\lambda$ satisfies ($A_n^\lambda$) or ($B_n^\lambda$), if its $R$-span does.
\end{nothing}

\begin{theorem}\label{exactness at -1}
The chain complex $\Ctilde_*^\lambda$ is exact in degree $-1$ for every composition $\lambda\in\Gamma$.
\end{theorem}

\begin{proof}
Since $S^\lambda$ has an $R$-complement in $M^\lambda$, cf.~(\ref{Specht property 3}), the restriction map $\Hom(M^\lambda,R)\to\Hom(S^\lambda,R)$ is surjective, and the proof is complete.
\end{proof}

The following lemma will be a key ingredient to the proof of Theorem~\ref{exactness at 0}, which in turn forms the base case for the inductive proof of Theorem~\ref{exactness for tame compositions}. To state the lemma we need to introduce the following notation. For $\lambda\in\Gamma$, every tableau $t\in\scrT(\lambda)$ can be thought of as a sequence of numbers by appending the second row at the end of the first, and so on. On these sequences we have the lexicographic total order. We set $t'<t$ if and only if the first entry in the sequence of $t$ which differs from the corresponding entry of $t'$ is greater than the one for $t'$. Moreover, for $t\in \scrTrs(\lambda)$ we denote by $C_{0,<t}^\lambda$ the $R$-span of the canonical basis elements $\varepsilon_{t'}$ of $\Hom(M^\lambda,R)$ with $t>t'\in\scrTrs(\lambda)$. See the end of Subsection~\ref{permutation module} for the definition of $\varepsilon_t$, $t\in\scrTrs(\lambda)$.

\begin{lemma}\label{key lemma}
Let $\lambda\in\Gamma$ be a composition and let $t\in\scrTrs(\lambda)$. If $t$ is not standard then $\varepsilon_t\in\im(d_1^\lambda)+C_{0,<t}^\lambda$.
\end{lemma}

The proof of Lemma~\ref{key lemma} will be given in the next section.

\begin{theorem}\label{exactness at 0}
Let $\lambda\in\Gamma$ be a composition and set
\begin{equation*}
  \scrK_0^\lambda:=\{(\lambda,t)\in\scrB_0^\lambda\mid t\in\scrTst(\lambda)\}\,.
\end{equation*}
Then ($A^\lambda_0$) and ($B_0^\lambda$) are satisfied. In particular, $\Ctilde_*^\lambda$ is exact in degree $0$ for every composition $\lambda\in\Gamma$.
\end{theorem}

\begin{proof}
In order to show that ($A_0^\lambda$) holds for $\scrK_0^\lambda$, we order the elements in $\scrTrs(\lambda)$ in ascending lexicographic order: $t_1<t_2<\cdots<t_n$. For every $i\in\{1,\ldots,n\}$, Lemma~\ref{key lemma} implies that $t_i\in\scrTst(\lambda)$ or $\varepsilon_{t_i}\in\im(d_1^\lambda)+C_{0,<t_i}^\lambda$. An easy induction argument on $i\in\{1,\ldots,n\}$ now shows that $\varepsilon_{t_i}\in\im(d_i^\lambda)+K_0^\lambda$, and ($A_0^\lambda$) holds.

If $\lambda$ is not a partition then ($B_0^\lambda$) holds trivially. Suppose that $\lambda$ is a partition.
In order to show that ($B_0^\lambda$) holds for $\scrK^\lambda_0$, we number the standard $\lambda$-tableaux $t_1,\ldots, t_m$ in such a way that the following holds for $i,j\in\{1,\ldots,m\}$: If $t_i=t^\lambda w_\lambda d_i$ with unique $w_{\lambda'}\le d_i\in W$, cf.~(\ref{weak Bruhat 1}) and (\ref{weak Bruhat 2}), then $l(w_\lambda d_i)<l(w_\lambda d_j)$ implies $i>j$. Next assume that there exists $0\neq\varepsilon\in K_0^\lambda\cap\ker(d_0^\lambda)$ and write $\varepsilon =\sum_{i=1}^m a_i\varepsilon_{t_i}$ with $a_i\in R$. Let $i$ be minimal with $a_i\neq 0$. Then, since $0= d_0^\lambda(\varepsilon)=\varepsilon|_{S^\lambda}$, we have
\begin{equation*}
  0=\varepsilon(z_\lambda T_{d_i}) = \sum_{j=i}^m a_j\varepsilon_{t_j}(z_\lambda T_{d_i})\,.
\end{equation*}
By (\ref{Specht property 1}) we obtain $\varepsilon_{t_i}(z_\lambda T_{d_i}) = q^{l(d_i)}$. Moreover, by (\ref{Specht property 2}) we obtain $\varepsilon_{t_j}(z_\lambda T_{d_i}) = 0$ for $j>i$. This implies $0=a_iq^{l(d_i)}$, and since $q$ is a unit, we have $a_i=0$, a contradiction. Thus, ($B_0^\lambda$) holds and the proof is complete.
\end{proof}

We recall from \cite{BH} that a composition $\mu\in\Gamma$ is called a {\em quasi-partition} if $\mu^*=\bar{\mu}$. Here, $\bar{\mu}$ is the unique smallest partition in $\Lambda$ which dominates $\mu$ (cf.~\cite[Lemma~4.2]{BH}) and $\mu^*\in\Lambda$ denotes the unique partition obtained by reordering the parts $\mu_1,\mu_2,\ldots$ of $\mu$ in weakly descending order. Recall also that a composition $\lambda\in\Gamma$ is called {\em tame} if every composition $\mu\in\Gamma$ with $\lambda\trianglelefteq \mu$ is a {\em quasi-partition}. Tame compositions are classified in \cite[Proposition~4.6]{BH}. They include all partitions of the form $(\lambda_1,\lambda_2,\lambda_3)$ with $\lambda_3\in\{0,1\}$. 

\smallskip
The technical definition of a tame composition is dictated by the inductive proof of the following theorem. The proof is a straight forward adaptation of the proof of Theorem~5.1 in \cite{BH} and will not be repeated here.

\begin{theorem}\label{exactness for tame compositions}
Let $\lambda\in\Gamma$ be a tame composition. For $n=-1$ set $\scrK_{-1}^\lambda:=\emptyset$, for $n=0$ set $\scrK_0^\lambda:=\scrTst(\lambda)$ and for $n\ge 1$ set
\begin{equation*}
  \scrK_n^\lambda:=\bigl\{(\lambda = \lambda^{(0)}\mathop{\triangleleft}\limits_{T_1}
  \lambda^{(1)}\mathop{\triangleleft}\limits_{T_2}\cdots
  \mathop{\triangleleft}\limits_{T_n}\lambda^{(n)}, t) \in\scrB_n^\lambda \mid t\in\scrTst(\lambda^{(n)}\bigr\}\,.
\end{equation*}
Then ($A_n^\lambda$) and ($B_n^\lambda$) hold for all $n\ge-1$. In particular, the chain complex $\Ctilde_*^\lambda$ is exact.
\end{theorem}

%%%%%%%%%%%%%%%%%% SECTION 5 %%%%%%%%%%%%%%%%%%%%%%%%%%%%%%%%%%%

\section{Proof of Lemma~\ref{key lemma}}

The exclusive goal of this section is to prove Lemma~\ref{key lemma}. The proof is a substantial refinement of the proof of \cite[Lemma~3.4]{BH}. Throughout we fix a composition $\lambda\in\Gamma$ and a $\lambda$-tableau $t\in\scrTrs(\lambda)\smallsetminus\scrTst(\lambda)$. We first need to establish some notation that will be used throughout this section.

\begin{nothing}
Since $t\in\scrTrs(\lambda)$ is not standard, there exist consecutive rows of $t$, say rows $k$ and $k+1$, which we write as
\begin{equation*}
  a_1\ a_2\ \cdots\ a_m\quad (m=\lambda_k)
\end{equation*}
and
\begin{equation*}
  b_1\ b_2\ \cdots\ b_n\quad (n=\lambda_{k+1}),
\end{equation*}
such that there exist $i\in\{0,\ldots,\min\{m,n-1\}\}$ satisfying
\begin{equation*}
  a_1<b_1,\ a_2<b_2,\ \cdots, a_i<b_i,\quad\text{and}\quad (a_{i+1}>b_{i+1}\ \text{or}\ m=i)\,.
\end{equation*}
The case $m=i$ can arise when $\lambda_{k+1}>\lambda_k$. We also set
\begin{equation*}
  Z:=\{a_1,\ldots,a_m,b_1,\ldots,b_n\}\,.
\end{equation*}
For every subset $X\subseteq Z$ with $|X|\le n$ we set 
\begin{equation*}
  \mu_X:=(\lambda_1,\ldots,\lambda_{k-1},\lambda_k+n-|X|, |X|,\lambda_{k+2},\ldots)\in\Gamma
\end{equation*}
and denote by $t_X\in\scrTrs(\mu_X)$ the tableau which coincides with $t$ in all rows except row $k$ and row $k+1$, and which has the elements of $X$ in row $k+1$ and those of $Z\smallsetminus X$ in row $k$. Moreover, we denote by $d_X\in\scrD_{\mu_X}$ the element which satisfies $t_X=t^{\mu_X}d_X$, and by $\varepsilon_X\in\Hom(M^{\mu_X},R)$ the element $\varepsilon_{t_X}=\varepsilon_{d_X}$. We will also use the composition
\begin{equation*}
  \nu_X:=(\lambda_1,\ldots,\lambda_k, n-|X|,|X|,\lambda_{k+2},\ldots)\in\Gamma\,.
\end{equation*}
We then have
\begin{equation*}
  \nu_X\trianglelefteq\lambda\trianglelefteq\mu_X\quad\text{and}\quad W_{\nu_X}=W_{\mu_X}\cap W_\lambda\,.
\end{equation*}
Finally, we define the adjacent, pairwise disjoint intervals of integers,
\begin{align*}
  A_X & :=\{\lambda_1+\cdots+\lambda_{k-1}+1,\lambda_1+\cdots+\lambda_{k-1}+2,\ldots,
  \lambda_1+\cdots+\lambda_{k-1}+\lambda_k\}\,,\\
  B_X & :=\{\lambda_1+\cdots+\lambda_k+1,\ldots,\lambda_1+\cdots+\lambda_k+(n-|X|)\}\,,\\
  C_X & :=\{\lambda_1+\cdots+\lambda_k+(n-|X|+1),\ldots,\lambda_1+\cdots+\lambda_k+n\}\,,
\end{align*}
(with $n=\lambda_{k+1}$)  and set 
\begin{equation*}
  D:=\{1,\ldots,r\}\smallsetminus (A_X\cup B_X\cup C_X)
\end{equation*}
Thus, $(A_X,B_X,C_X)$ are the $k$-th, $(k+1)$-th, $(k+2)$-th respective subsets of the set partition associated with $\nu_X$, $(A_X, B_X\cup C_X)$ are the $k$-th and $(k+1)$-th respective subsets of the set partition associated with $\lambda$, and $(A_X\cup B_X,C_X)$ are the $k$-th and $(k+1)$-th subsets associated with $\mu_X$.

Recall from Subsection~\ref{Hom(M,M)} that, for $X$ as above, one has an element $\varphi_1^{\mu_X,\lambda}\in\HomhatH(M^\lambda,M^{\mu_X})$.
\end{nothing}

\begin{lemma}\label{lemma 1}
Let $X\subseteq Z$ be a subset with $|X|\le n$.

{\rm (a)} For $d\in\scrD_\lambda$, the set
\begin{equation*}
  W_{X,d}:= W_{\mu_X}\cap(\scrD_{\nu_X}\cap W_\lambda)dd_X^{-1}
\end{equation*}
has at most one element.

{\rm (b)} In $\Hom(M^\lambda,R)$, one has the equation
\begin{equation*}
  \varepsilon_X\circ\varphi_1^{\mu_X,\lambda} =
  \sum_{d\in\scrD_\lambda \atop W_{X,d}\neq\emptyset} q^{l(w_{X,d})}\varepsilon_d\,,
\end{equation*}
where we write $W_{X,d}=\{w_{X,d}\}$ when $W_{X,d}\neq\emptyset$.
\end{lemma}

\begin{proof}
(a) Let $d\in\scrD_\lambda$ and assume that one has elements $w_1,w_2\in W_{\mu_X}$ and $e_1,e_2\in\scrD_{\nu_X}\cap W_\lambda$ with $w_i=e_idd_X^{-1}$ for $i=1,2$. Then $e_1e_2^{-1}=w_1w_2^{-1}\in W_{\mu_X}\cap W_\lambda=W_{\nu_X}$. Since $e_1,e_2\in\scrD_{\nu_X}$, we have $e_1=e_2$ and $w_1=w_2$.

(b) Recall that the elements $x_\lambda T_d$, $d\in\scrD_\lambda$, form an $R$-basis of $M^\lambda$. For $d\in\scrD_\lambda$ we have
\begin{equation*}
  \varepsilon_X\bigl(\varphi_1^{\mu_X,\lambda}(x_\lambda T_d)\bigr) 
  = \varepsilon_X\bigl(\varphi_1^{\mu_X,\lambda}(x_\lambda)T_d\bigr)
  = \sum_{e\in\scrD_{\nu_X}\cap W_\lambda} \varepsilon_X(x_{\mu_X} T_{ed})
\end{equation*}
by (\ref{explicit permutation homomorphism}). For every $e\in\scrD_{\nu_X}\cap W_\lambda$ there exists a unique $f_{d,e}\in\scrD_{\mu_X}$ with $ed\in W_{\mu_X}f_{e,d}$ and a unique element $w_{e,d}\in W_{\mu_X}$ with $ed=w_{e,d}f_{e,d}$. We continue the above computation:
\begin{equation*}
\begin{split}
  \varepsilon_X & \bigl(\varphi_1^{\mu_X,\lambda}(x_\lambda T_d)\bigr)
  = \sum_{e\in\scrD_{\nu_X}\cap W_\lambda} \varepsilon_X(x_{\mu_X} T_{w_{e,d} f_{e,d}}) \\
  & = \sum_{e\in\scrD_{\nu_X}\cap W_\lambda} q^{l(w_{e,d})} \varepsilon_X(x_{\mu_X} T_{f_{e,d}})
  = \sum_{e\in\scrD_{\nu_X}\cap W_\lambda \atop f_{e,d} = d_X} q^{l(w_{e,d})}\,.
\end{split}
\end{equation*}
But, if $e\in\scrD_{\nu_X}\cap W_\lambda$ satisfies $f_{e,d} = d_X$ then $w_{e,d} = edd_X^{-1} \in W_{X,d} = 
\{w_{X,d}\}$ and $e$ is the only element in $\scrD_{\nu_X}\cap W_\lambda$ with $f_{e,d} = d_X$. Conversely, if $W_{X,d}\neq\emptyset$ and we write $w_{X,d} = edd_X^{-1}$ with $e\in\scrD_{\nu_X}\cap W_\lambda$ then $e$ satisfies $f_{e,d}=d_X$. Thus, we have
\begin{equation*}
  \varepsilon_X\bigl( \varphi_1^{\mu_X,\lambda}(x_\lambda T_d)\bigr) = 
  \begin{cases}
    0 & \text{if $W_{X,d} = \emptyset$,}\\
    q^{l(w_{X,d})} & \text{if $W_{X,d} \neq \emptyset$.}
  \end{cases}
\end{equation*}
Now the statement in (b) is immediate and the proof is complete.
\end{proof}

\begin{lemma}\label{lemma 2}
Let $X\subseteq Z$ be a subset with $|X|\le n$.

{\rm (a)} The function
\begin{equation*}
  \{X\subseteq Y\subseteq Z\mid |Y|=n\}\to\scrD_\lambda\,,\quad Y\mapsto d_Y\,,
\end{equation*}
is injective with image $\{d\in\scrD_\lambda\mid W_{X,d}\neq\emptyset\}$.

{\rm (b)} In $\Hom(M^\lambda, R)$, one has
\begin{equation*}
  \varepsilon_X\circ\varphi_1^{\mu_X,\lambda} = \sum_{X\subseteq Y\subseteq Z \atop |Y|=n}
  q^{l(w_{X,Y})} \varepsilon_Y\,,
\end{equation*}
where $w_{X,Y}:=w_{X,d_Y}$.

{\rm (c)} For $X\subseteq Y\subseteq Z$ with $|Y|=n$ one has
\begin{equation*}
  l(w_{X,Y})=\sum_{y\in Y\smallsetminus X} |\{z\in Z\smallsetminus Y \mid y<z\}|\,.
\end{equation*}
\end{lemma}

\begin{proof}
(a) For $X\subseteq Y\subseteq Z$ with $|Y|=n$ we have $d_Y\in\scrD_\lambda$, since $\mu_Y=\lambda$. Moreover, the map $Y\mapsto d_Y$ is clearly injective. Our next goal is to show that $W_{X,d_Y}\neq\emptyset$. Note that $X\subseteq Y=(B_X\cup C_X)d_Y$ implies $(X)d_Y^{-1} \subseteq B_X \cup C_X$. We define an element $e\in W$ by
\begin{align}\label{e definition}
\begin{split}
  & e|_{A_X\cup D}\text{ is the identity,}\\
  & e|_{C_X}\colon C_X\liso (X)d_Y^{-1} \text{ is monotonous,}\\
  & e|_{B_X}\colon B_X\liso (B_X\cup C_X)\smallsetminus (X)d_Y^{-1} \text{ is monotonous.}
\end{split}
\end{align}
Then, $e\in\scrD_{\nu_X}$, since $e$ is monotonous on $A_X\cup B_X$, on $C_X$ and on $D$. Moreover, $e\in W_\lambda$, since $(B_X\cup C_X)e=B_X\cup C_X$ and $e$ is the identity on $A_X\cup D$. We set
\begin{equation}\label{w definition}
  w:=ed_Yd_X^{-1}\,.
\end{equation}
We claim that
\begin{equation}\label{w is identity}
  w|_{C_X\cup D} \text{ is the identity map.}
\end{equation}
In fact, for $x\in D$ we have $(x)ed_Yd_X^{-1} =(x)d_Yd_X^{-1}=x$, since $d_X$ and $d_Y$ coincide on $D$. Also, $w|_{C_X}$ is the composition
\begin{equation*}
  C_X\Ar{e} (X)d_Y^{-1} \Ar{d_Y} X \Ar{d_X^{-1}} C_X\,.
\end{equation*}
This composition is monotonous, since each of the three factors is monotonous (because $e\in\scrD_{\nu_X}$, $d_Y\in\scrD_\lambda$, $d_X\in\scrD_{\mu_X}$ and $(C_X)e\subseteq B_X\cup C_X$). Thus, $w|_{C_X}$ is the identity map and Claim~(\ref{w is identity}) is proven. Since $w$ is the identity on $D\cup C_X$, we obtain $(A_X\cup B_X)w =A_X\cup B_X$. This implies $w\in W_{\mu_X}$ and $w\in W_{X,d_Y}$.

Conversely, let $d\in\scrD_\lambda$ with $W_{X,d}\neq\emptyset$. Then there exits $w\in W_{\mu_X}$ and $e\in\scrD_{\nu_X}\cap W_{\lambda}$ with $w=edd_X^{-1}$. Since $w\in W_{\mu_X}$, we have $(C_X)w=C_X$, and since $e\in W_{\lambda}$, we have $(C_X)e\subseteq(B_X\cup C_X)e=B_X\cup C_X$. Therefore,
\begin{equation*}
  Y:= (B_X\cup C_X)d \supseteq (C_X)ed = (C_X)wd_X = (C_X)d_X = X
\end{equation*}
and $|Y|=|B_X\cup C_X|=n$. We claim that $d=d_Y$. For $j\notin\{k,k+1\}$ and $D_j:=\{\lambda_1+\cdots+\lambda_{j-1}+1,\ldots,\lambda_1+\cdots+\lambda_{j-1}+\lambda_j\}$ we have
\begin{equation*}
  (D_j)d = (D_j)ed = (D_j)wd_X = (D_j) d_X = (D_j)d_Y\,,
\end{equation*}
since $e\in W_\lambda$ and $w\in W_{\mu_X}$. Moreover, $d$ and $d_Y$ are monotonous on $D_j$, since $d,d_Y\in\scrD_\lambda$. Therefore, $d$ and $d_Y$ coincide on $D_j$. Further, $(B_X\cup C_X) d_Y = Y = (B_X\cup C_X) d$ and $d$ and $d_Y$ are monotonous on $B_X\cup C_X$. Therefore, $d$ and $d_Y$ coincide on $B_X\cup C_X$. The above implies that $(A_X)d=(A_X)d_Y$ and since $d$ and $d_Y$ are monotonous on $A_X$, they also coincide on $A_X$. Thus, $d=d_Y$.

(b) Using Lemma~\ref{lemma 1}(b) and Part (a), we have
\begin{equation*}
  \varepsilon_X\circ\varphi_1^{\mu_X,\lambda} =
  \sum_{d\in\scrD_\lambda \atop W_{X,d}\neq\emptyset} q^{l(w_{X,d})}\varepsilon_d = 
  \sum_{X\subseteq Y\subseteq Z \atop |Y|=n} q^{l(w_{X,d_Y})} \varepsilon_{d_Y}
\end{equation*}
and (b) is proven.

(c) By (\ref{w is identity}), $w_{X,Y}$ is the identity on $D\cup C_X$, and we can view $w_{X,Y}$ as a permutation of $A_X\cup B_X$. We claim that $w:=w_{X,Y}$ is the permutation of $A_X\cup B_X$ with maps $B_X$ monotonously onto the positions of $Y\smallsetminus X$ in the $k$-th row of $t_X$. In fact, we need to show that $(B_X)wd_X = Y\smallsetminus X$ and that $w$ is monotonous on $B_X$. But by (\ref{e definition}) we have
\begin{equation*}
  (B_X)wd_X = (B_X)ed_Y = \bigl((B_X\cup C_X)\smallsetminus(X)d_Y^{-1}\bigr) d_Y = Y\smallsetminus X\,,
\end{equation*}
where $e$ is defined as in Part (a). Moreover, since $(B_X)w\subseteq A_X\cup B_X$ and since $d_X$ is monotonous on $A_X\cup B_X$, $w$ is monotonous on $B_X$ if and only if $wd_X=ed_Y$ is. Again, since $(B_X)e\subseteq B_X\cup C_X$ and $d_Y$ is monotonous on $B_X\cup C_X$, $ed_Y$ is monotonous on $B_X$ if and only if $e$ is. But, $e\in\scrD_{\mu_X}$ is monotonous on $A_X\cup B_X$ and so also on $B_X$. This shows the claim.

Since $l(w)$ is equal to the number of inversions of $w$, where $w$ is viewed as a permutation of $A_X\cup B_X$, our claim immediately implies the desired equation if we can also show that $w$ is monotonous on $A_X$. But, $w$ is monotonous on $A_X$ if and only if $wd_X=ed_Y$ is, by the argument above. Further, $e$ is the identity on $A_X$, cf.~(\ref{e definition}), so that with $d_Y$ also $ed_Y$ is monotonous on $A_X$, and the proof of the lemma is complete.
\end{proof}

Next we consider the element
\begin{equation*}
  c_1:=\sum_{X\subseteq \{a_1,\ldots,a_i\}} (-1)^{|X|} q^{f(X)}
  \bigl(\varphi_1^{\mu_{\Xtilde},\lambda} \otimes \varepsilon_{\Xtilde}\bigr) \in C_1^\lambda\,,
\end{equation*}
where
\begin{equation*}
  \Xtilde:= X\cup \{b_{i+2},\cdots,b_n\}
\end{equation*}
and
\begin{equation*}
  f(X):=\sum_{j\in\{1,\ldots,i\} \atop a_j\in X} (m+1-j)\,.
\end{equation*}

\begin{lemma}\label{lemma 3}
With the above notation, one has
\begin{equation*}
  d_1^\lambda(c_1)\in q^l \varepsilon_t + C_{0,<t}\,,
\end{equation*}
for some integer $l$.
\end{lemma}

\begin{proof}
By Lemma~\ref{lemma 2}(b) we have
\begin{equation*}
\begin{split}
  d_1^\lambda(c_1) = & 
  \sum_{X\subseteq \{a_1,\ldots,a_i\}} (-1)^{|X|} q^{f(X)} 
  \bigl(\varepsilon_{\Xtilde}\circ \varphi_1^{\mu_{\Xtilde},\lambda}\bigr) \\
  = & \sum_{X\subseteq\{a_1,\ldots,a_i\}} (-1)^{|X|} q^{f(X)}
  \sum_{\Xtilde\subseteq Y\subseteq Z \atop  |Y|=n} q^{l(w_{\Xtilde,Y})} \varepsilon_Y\\
  = & \sum_{\{b_{i+2},\ldots,b_n\}\subseteq Y\subseteq Z \atop |Y|=n}
  \Bigl(\sum_{X\subseteq \{a_1,\ldots,a_i\}\cap Y} (-1)^{|X|} q^{f(X)+l(w_{\Xtilde,Y})} \Bigr) \varepsilon_Y\,.
\end{split}
\end{equation*}
For every $Y\subseteq Z$ with $\{b_{i+2},\ldots,b_n\}\subseteq Y$ and $|Y|=n$ we study the contribution of the corresponding summand at the end of the last equation. We distinguish three cases.

If $Y=\{b_1,\ldots,b_n\}$ then $\varepsilon_Y=\varepsilon_t$ and the only possible set $X$ in the inner sum of the last equation is the empty set. Therefore, the summand corresponding to $Y$ is equal to $q^l\varepsilon_t$ with $l=f(\emptyset)+l(w_{\widetilde{\emptyset},Y})$.

If $Y\cap \{a_1,\ldots,a_i\}=\emptyset$ and $Y\neq\{b_1,\ldots,b_n\}$ we claim that $t_Y<t$. Note that in this case we have $m\ge 1$. Let $j\in\{1,\ldots,i+1\}$ be minimal with $b_j\notin Y$ (note that $a_1<b_j$) and let $p\in\{j-1,\ldots,i\}$ be maximal with $a_p<b_j$ (note that $a_{j-1}<b_{j-1}<b_j$). Then the $k$-th row of $t_Y$ begins with $a_1, a_2,\ldots, a_p,b_j$ and $b_j<a_{p+1}$. Note that $p<m$, since if $p=m$ then $a_1,\ldots,a_m\in Z\smallsetminus Y$ and therefore $Y\subseteq\{b_1,\ldots,b_n\}$, a contradiction. Now our claim is proved and the contribution of the summand corresponding to $Y$ is an element in $C^\lambda_{0,<t}$.

We are left with the case that $Y\cap\{a_1,\ldots,a_i\}\neq\emptyset$. If $t_Y<t$ we are done. So assume that $t_Y>t$. Our goal is to show that
\begin{equation*}
  \sum_{X\subseteq\{a_1,\ldots,a_i\}\cap Y}
  (-1)^{|X|} q^{f(X)+l(w_{\Xtilde,Y})} = 0\,.
\end{equation*}
First let $j\in\{1,\ldots, i\}$ be minimal with $a_j\in Y$. We claim that
\begin{equation}\label{extremal property}
  \{z\in Z\smallsetminus Y \mid z<a_j\} = \{a_1,\ldots,a_{j-1}\}\,.
\end{equation}
Clearly, the right hand side is contained in the left hand side. Now let $z\in Z\smallsetminus Y$ with $z<a_j$ and assume that $z$ is not contained in the right hand side. Then $z=b_p$ for some $p\in\{1,\ldots, i+1\}$, since $b_{i+2},\ldots,b_n\in Y$. Let $r\in\{1,\ldots,i+1\}$ be minimal with $b_r\in Z\smallsetminus Y$ and $b_r<a_j$, and let $s\in\{1,\ldots,j-1\}$ be maximal with $a_s<b_r$ (note that $j>1$, since $b_r<a_j$ and that $a_1<b_r$). Then the $k$-th row of $t_Y$ starts with $a_1,a_2,\ldots,a_s,b_r$ and one has $b_r<a_{s+1}$. This implies $t_Y<t$, a contradiction. Therefore, the claim (\ref{extremal property}) is proven. For our given set $Y$, the subsets $X$ of $\{a_1,\ldots,a_i\}\cap Y =\{a_j,\ldots,a_i\}\cap Y$ fall into two classes, the ones that contain $a_j$ and the ones that don't. So let $X_1$ be a subset of $\{a_1,\ldots,a_i\}\cap Y$ with $a_j\in X_1$ and let $X_2:=X_1\smallsetminus \{a_j\}$. Since $X_1\mapsto X_2$ defines a bijection between these two classes, it suffices to show that 
\begin{equation*}
  (-1)^{|X_1|} q^{f(X_1)+l(w_{\Xtilde_1,Y})} + (-1)^{|X_2|} q^{f(X_2)+l(w_{\Xtilde_2,Y})} =0\,,
\end{equation*}
or, equivalently, that
\begin{equation}\label{difference}
  l(w_{\Xtilde_1,Y})+m+1-j = l(w_{\Xtilde_2,Y})\,.
\end{equation}
But by Lemma~\ref{lemma 2}(c) and Equation~(\ref{extremal property}), and since $Y\smallsetminus \Xtilde_2 = Y\smallsetminus \Xtilde_1 \cup \{a_j\}$ and $|Z\smallsetminus Y|=m$, we have
\begin{align*}
  & \quad l(w_{\Xtilde_2,Y}) - l(w_{\Xtilde_1,Y}) \\
  = \  & \sum_{y\in Y\smallsetminus \Xtilde_2} |\{z\in Z\smallsetminus Y \mid y<z\}| -
  \sum_{y\in Y\smallsetminus \Xtilde_1} |\{z\in Z\smallsetminus Y \mid y<z\}|\\
  = \ & |\{z\in Z\smallsetminus Y \mid a_j<z\}| = m - |\{z\in Z\smallsetminus Y \mid z<a_j\}| = m-(j-1)\,.
\end{align*}
This shows Equation~(\ref{difference}) and the proof of the lemma is complete.  
\end{proof}

Now Lemma~\ref{key lemma} is an immediate consequence of Lemma~\ref{lemma 3}.

%%%%%%%%%%%%%%% BIBLIOGRAPHY %%%%%%%%%%%%%%%%%%%%%%%%%%%%%%%

\end{document}